\documentclass{amsart}[12 pt]

\usepackage[utf8]{inputenc} 
\usepackage[T1]{fontenc} 

\usepackage{amsmath, amsthm, amssymb, mathrsfs,amsrefs}
\usepackage{enumerate}
\usepackage{tikz-cd}
\usetikzlibrary{arrows}

\usetikzlibrary{matrix,arrows,decorations.pathmorphing}

 \let\temp\phi
\let\phi\varphi
\let\varphi\temp

\newcommand{\acts}{\curvearrowright}

\newcommand{\cl}[1]{\overline{#1}}

\newcommand{\C}{\mathbb{C}}

\newcommand{\N}{\mathbb{N}}

\newcommand{\T}{\mathbb{T}}

\newcommand{\Z}{\mathbb{Z}}

\DeclareMathOperator{\Span}{span}

\newcommand{\calF}{\mathcal{F}}

\newcommand{\calJ}{\mathcal{J}}
\newcommand{\calK}{\mathcal{K}}
\newcommand{\calL}{\mathcal{L}}
\newcommand{\calM}{\mathcal{M}}

\newcommand{\calO}{\mathcal{O}}

\newcommand{\calS}{\mathcal{S}}
\newcommand{\calT}{\mathcal{T}}
\newcommand{\calU}{\mathcal{U}}

\newcommand{\Angle}[1]{\left\langle #1 \right\rangle}

\renewcommand{\Hat}{\widehat}

\DeclareMathOperator{\fin}{fin.}

\theoremstyle{plain}
\newtheorem{lemma}{Lemma}[section]
\newtheorem{theorem}[lemma]{Theorem}

\newtheorem{corollary}[lemma]{Corollary}

\theoremstyle{definition}

\newtheorem{example}[lemma]{Example}

\newtheorem{definition}[lemma]{Definition}

\DeclareMathOperator{\Int}{int}

\title{Hyperrigidity of C*-correspondences}
\author{Se-Jin Kim}
\address{University of Waterloo \\
Department of Pure Mathematics \\
Waterloo, Ontario \\
Canada  N2L 3G1}
\email{s362kim@uwaterloo.ca}

\begin{document}

\begin{abstract}
  We show that hyperrigidity for a C*-correspondence $(A,X)$ is equivalent to non-degeneracy of the left action of the Katsura ideal $\calJ_X$ on $X$. Due to the work of Katsoulis and Ramsey~\cite{KatRam2}, our result shows that if $G$ is a locally compact group acting on $(A,X)$ and the Katsura ideal $\calJ_X$ acts on $X$ non-degenerately then the Hao-Ng isomorphism problem for reduced crossed products has a positive solution and the Hao-Ng isomorphism problem for full crossed products has a partial solution.
\end{abstract}

\maketitle

\section{Introduction}
	This paper is concerned with the following isomorphism problem: if $(A,X)$ is a (non-degenerate) C*-correspondence and $G$ is a locally compact group acting continuously on $(A,X)$ then is it the case that we have the identity
	\begin{align*}
		\calO_X \rtimes G = \calO_{X \rtimes G}\;?
	\end{align*}
	The algebra $\calO_X$ is the Cuntz-Pimsner algebra associated to the C*-correspondence $(A,X)$ and $X \rtimes G$ is the crossed product C*-correspondence over $A \rtimes G$. This problem is called the \emph{Hao-Ng isomorphism problem}. It is named after Hao and Ng who establish this identity for the case when $G$ is a locally compact and amenable group~\cite{HaoNg}. At its core, the Hao-Ng isomorphism problem is asking whether the functor which maps a C*-correspondence $(A,X)$ to its Cuntz-Pimsner algebra $\calO_X$ is closed under crossed products. Because of this, the isomorphism problem is fundamental in the understanding of the dynamics of Cuntz-Pimsner algebras.  Recently, significant progress has been made by Katsoulis \cite{Kat} and Katsoulis and Ramsey \cite{KatRam1, KatRam2} on the isomorphism problem using Arveson's notion of hyperrigidity \cite{Arveson2}. In this paper we establish an intrinsic  characterization of hyperrigidity for C*-correspondences.

	We say that a C*-correspondence $(A,X)$ is hyperrigid if the operator space
	\begin{align*}
		S(A,X):= \Span\{x+a+y^*: x,y \in X, a \in A \} \subset \calO_X
	\end{align*}
	has the following extension property: given a representation $\pi: \calO_X \to B(H)$, if $\phi: \calO_X \to B(H)$ is a completely positive and completely contractive map which agrees with $\pi$ on $\calS(A,X)$ then $\phi$ must agree with $\pi$ on $\calO_X$. In \cite{KatRam2}, Katsoulis and Ramsey establish:
	\begin{enumerate}
	\item If $(A,X)$ is a hyperrigid C*-correspondence and $G$ is a locally compact group that acts on $(A,X)$ then we have the identity
	\begin{align*}
		\calO_X \rtimes G = \calO_{X \Hat{\rtimes} G}
	\end{align*}
	where $(A \Hat{\rtimes} G, X \Hat{\rtimes} G)$ is the completion of the pair $(C_c(G,A), C_c(G,X))$ in $\calO_X \rtimes G$. In particular, for hyperrigid C*-correspondences, the crossed product $\calO_X \rtimes G$ is a Cuntz-Pimsner algebra.

	\item If $(A,X)$ is a hyperrigid C*-correspondence and $G$ is a locally compact group that acts on $(A,X)$ then we have the identity
	\begin{align*}
		\calO_X \rtimes_r G = \calO_{X \rtimes_r G}\;.
	\end{align*}
	\end{enumerate}
  It is an open question whether $(A \Hat{\rtimes} G, X \Hat{\rtimes} G)$ is the same as $(A \rtimes G, X \rtimes G)$.

	Our main Theorem is the following.
	\begin{theorem}
		Let $(A,X)$ be a C*-correspondence. The following are equivalent:
		\begin{enumerate}
			\item The C*-correspondence $(A,X)$ is hyperrigid.
			\item We have the identity $\calJ_X \cdot X = X$.
		\end{enumerate}
	\end{theorem}
	This extends a result of Kakariadis~\cite[Theorem 3.3]{Kakariadis} and Dor-On and Salomon~\cite[Theorem 3.5]{DorOnSalomon} who establish the equivalence for C*-correspondences associated to discrete graphs and a result of Katsoulis and Ramsey who give a sufficient condition for hyperrigidity when $X$ is countably generated over $A$ \cite[Theorem 3.1]{KatRam2}. Finally, we use this result to give an exact characterization for when the C*-correspondence associated to a topological graph is hyperrigid when the range map $r$ is open.

\section{Preliminaries}
In this section, we give a brief overview of the various results on operator systems and Cuntz-Pimsner algebras that we will need for this paper.

\subsection{Operator systems}
	An operator system $S$ is a closed subspace of a unital C*-algebra $A$ for which $1_A \in S$ and $S^* = S$. The class of operator systems has an abstract axiomatization \cite{ChoiEffros}. We will only say a word about the abstract characterization: to axiomatize operator systems it is enough to keep track of the involution $^*$, the cone $M_n(S)_+$ of positive operators on $M_n(S) \subset M_n(A)$, and the unit $1_{M_n(A)} \in M_n(S)$. The appropriate morphisms for operator systems are unital completely positive (ucp) maps and the appropriate embeddings for operator systems are unital complete order isometries. Given an operator system $S$, we say that a pair $(C,\rho)$ is a C*-cover of $S$ if $C$ is a C*-algebra and $\rho$ is a unital complete order isometry $\rho: S \hookrightarrow C$ for which $C^*(\rho(S)) = C$. Given an operator system $S$ there is always a minimal C*-cover called the C*-envelope $(C^*_e(S),\iota)$. It is minimal in the following sense: if $(C,\rho)$ is another C*-cover of $S$ then there is a *-homomorphism $\pi: C \to C^*_e(S)$ for which the diagram
	\begin{center}
		\begin{tikzcd}
			C \arrow{rd}{\pi} & \\
			S \arrow{r}{\iota} \arrow{u}{\rho} & C^*_e(S)
		\end{tikzcd}
	\end{center}
commutes. There is also a universal C*-cover $(C^*_{\max}(S),\iota)$ which is maximal in the following sense: if $(C,\rho)$ is another C*-cover of $S$ then there is a *-homomorphism $\pi: C^*_{\max}(S) \to C$ for which the diagram
	\begin{center}
		\begin{tikzcd}
			C^*_{\max}(S) \arrow{rd}{\pi} & \\
			S \arrow{r}{\rho} \arrow{u}{\iota} & C
		\end{tikzcd}
	\end{center}
commutes. We will always assume without loss of generality that $S \subset C^*_e(S)$.

	An operator subsystem $S$ of a C*-algebra $A$ is said to be hyperrigid in $A$ if we have the following unique extension property: whenever $\pi: C^*(S) \to B(H)$ is a *-homomorphism and whenever $\phi: C^*(S) \to B(H)$ is a unital completely positive (ucp) map extending the ucp map $\pi|_S$ then we must have $\phi = \pi$. Hyperrigid operator systems give us a strong relation between operator systems and their C*-envelope. For example, if $S$ is hyperrigid in $A$ then we must have $C^*(S) \simeq C^*_e(S)$. We say that $S$ is hyperrigid if $S$ is hyperrigid in $C^*_e(S)$. The above definition of hyperrigidity is not the original one. In \cite[Definition 1.1]{Arveson2}, a subspace $S \subset A$ is said to be hyperrigid if whenever we have a faithful embedding $A \subset B(H)$ and whenever $\phi_n: B(H) \to B(H)$ is a sequence of completely contractive and completely positive maps, we have the implication
  \begin{align*}
    \lim_{n \to \infty} \|\phi_n(x)- x\| = 0 \text{ for all }x \in S \text{ implies } \lim_{n \to \infty} \|\phi_n(a) - a\| = 0 \text{ for all }a \in A\;.
  \end{align*}
  In \cite[Theorem 2.1]{Arveson2}, Arveson proves that these two definitions are equivalent in the separable case. The density character of a topological space $X$ is the smallest cardinal $\kappa$ for which there is a subset $E \subset X$ of size $\kappa$ that is dense in $X$. Arveson's proof will go through verbatim when we replace all instances of \emph{separable} with \emph{density character at most $\kappa$} for any infinite cardinal $\kappa$.

	If $S$ is *-closed but non-unital, so long as $S$ contains an approximate unit of $A$, it follows from \cite[Proposition 3.6]{Salomon} that $S$ is hyperrigid in $C^*(S)$ if and only if $S^1 := S+\C 1$ in the unitization $C^*(S)^1$ is hyperrigid.

	A representation $\pi: C^*(S) \to B(H)$ is said to be boundary if $\pi$ is irreducible and $\pi$ admits the unique extension property. Arveson's hyperrigidity conjecture asserts that if all irreducible representations are boundary then the operator system $S$ must be hyperrigid in $A$. Very little is known about the hyperrigidity conjecture. For more information on operator systems, see \cite{Paulsen}. See \cite{Arveson2} for the formulation of the hyperrigidity conjecture and more details on the above results.

  \subsection{The tensor algebra $\calT^+_X$}

Let $(A,X)$ be a C*-correspondence and let $C$ be a C*-algebra. We say that a pair of maps $(\pi^0,\pi^1): (A,X) \to C$ is a Toeplitz pair if
\begin{enumerate}
	\item $\pi^0: A \to C$ is a *-homomorphism,
	\item $\pi^1: X \to C$ is a linear map,
	\item For any $a \in A$ and $x \in X$ we have $\pi^0(a)\pi^1(x) = \pi^1(a\cdot x)$, and
	\item For any $x$ and $y$ in $X$ we have $\pi^0(\Angle{x,y}) = \pi^1(x)^* \pi^1(y)$.
\end{enumerate}

Given a Toeplitz pair $(\pi^0,\pi^1)$, we always have $\pi^1(x)\pi^0(a) = \pi^1(x \cdot a)$ for any $x \in X$ and $a \in A$. A Toeplitz pair can also be thought of as a morphism from the C*-correspondence $(A,X)$ into the C*-correspondence $(C,C)$ where left and right action is given by multiplication and the inner product is given by $\Angle{x,y} = x^*y$. There is always a maximal C*-algebra associated to C*-correspondences called the Toeplitz-Pimsner algebra $\calT_X$. This C*-algebra is maximal in the following sense: there is always a Toeplitz pair
\begin{align*}
	\kappa^0 &: A \to \calT_X \\
	\kappa^1 &: X \to \calT_X
\end{align*}
into $\calT_X$ and whenever $(\pi^0, \pi^1) :(A,X) \to C$ is a Toeplitz pair then there is a *-homomorphism
\begin{align*}
	\pi^0 \times \pi^1 : \calT_X \to C
\end{align*}
for which the diagram
\begin{center}
	\begin{tikzcd}[row sep= large, column sep=large]
		\calT_X \arrow{r}{\pi^0 \times \pi^1} & C \\
		(A,X) \arrow{u}{(\kappa^0,\kappa^1)} \arrow{ur}{(\pi^0,\pi^1)} &
	\end{tikzcd}
\end{center}
commutes. The Toeplitz-Pimsner algebra always contains a canonical norm closed non-selfadjoint operator algebra $\calT^+_X$ called the Tensor algebra. This algebra is described as the non-selfadjoint operator algebra generated by $\kappa^0(A)$ and $\kappa^1(X)$ in $\calT_X$.

The Toeplitz-Pimsner algebra $\calT_X$ always admits a canonical continuous $\T$-action $\gamma$ called the gauge action. Using the universal property of $\calT_X$, it is enough to define $\gamma$ as an action on $(A,X)$: for $z \in \T$,
\begin{align*}
	\gamma^0_z &: A \to A : a \mapsto a \\
	\gamma^1_z &: X \to X : x \mapsto z\cdot x
\end{align*}
will give us the action.

Although the Toeplitz-Pimsner algebra $\calT_X$ is a canonical algebra associated to $(A,X)$, it is often too big for our purposes. As an example, the gauge-invariant uniqueness theorem for graph algebras will not generalize to $\calT_X$. For example, the algebra $\calT_{\C}$ is the universal C*-algebra generated by a single isometry. On the other hand, there is always a gauge-invariant Toeplitz pair from $(\C,\C)$ into $C^*(\Z)$ by mapping $1$ to the canonical unitary associated to $1 \in \Z$. The remedy for this is to restrict our class of representations.

	Fix a C*-correspondence $(A,X)$. The compact operators $\calK(X)$ is the C*-subalgebra of the space $\calL(X)$ of adjointable right-$A$-linear operators on $X$ spanned by the rank one operators $x \Angle{y,\cdot}$ for $x,y \in X$. We think of the left action of $A$ on $X$ as the *-homomorphism $\lambda: A \to \calL(X)$. Given a Toeplitz pair $(\pi^0, \pi^1): (A,X) \to C$, there is always a *-homomorphism
	\begin{align*}
		\phi_{\pi} : \calK(X) \to C: x\Angle{y,\cdot} \mapsto \pi^1(x)\pi^1(y)^*\;.
	\end{align*}

	The Katsura ideal $\calJ_X$ associated to $(A,X)$ consists of elements $a \in A$ for which $\lambda(a) \in \calK(X)$ and for which $ab=0$ whenever $b$ belongs to the kernel of $\lambda$. A Toeplitz pair $(\pi^0,\pi^1): (A,X) \to C$ is said to be covariant if for any element $a \in \calJ_X$, we have the identity
	\begin{align*}
		\pi^0(a) = \phi_\pi(\lambda(a)) \;.
	\end{align*}
	The appropriate choice of C*-algebra is the universal C*-algebra associated to covariant Toeplitz pairs. This algebra is called the Cuntz-Pimsner algebra $\calO_X$. We will let
	\begin{align*}
		\iota^0 &: A \to \calO_X \\
		\iota^1 &: X \to \calO_X
	\end{align*}
	be the canonical covariant Toeplitz pair. Since the gauge action $(\gamma^0,\gamma^1) : \T \acts (A,X)$ is covariant, $\calO_X$ has a gauge action as well. As well, there is a canonical quotient map $\calT_X \to \calO_X$. We also have the gauge invariant uniqueness theorem~\cite[Theorem 6.4]{Katsura}.

	\begin{theorem}[Gauge-invariant uniqueness theorem]
		Suppose that there is a covariant Toeplitz pair $(\pi^0, \pi^1): (A,X) \to C$ with $\pi^0$ injective and suppose that there is a gauge action $\T \acts C^*(\pi^0,\pi^1)$ for which the Toeplitz pair $(\pi^0, \pi^1)$ is $\T$-equivariant. The *-homomorphism
		\begin{align*}
			\pi^0 \times \pi^1 : \calO_X \to C
		\end{align*}
		is necessarily injective.
	\end{theorem}

	\begin{example}
		Let $E = (E^0, E^1, s,r)$ be a topological graph. We will always assume that our topological graphs are $r$-discrete. Define a C*-correspondence $X(E)$ over the C*-algebra $C_0(E^0)$ as the completion of $C_c(E^1)$ with left and right actions given by
		\begin{align*}
			f \cdot g &: e \mapsto  f(e) g(s(e)) \text{ and} \\
			g \cdot f &: e \mapsto g(r(e))f(e)
		\end{align*}
		for any $f\in C_c(E^1)$ and $g \in C_0(E^0)$ and with inner product given by
		\begin{align*}
			\Angle{f,h} : x \in E^0 \mapsto \sum_{e \in E^1 : s(e) = x} \cl{f(e)}h(e)
		\end{align*}
		for any $f,h \in C_c(E^1)$. The graph C*-algebra $C^*(E)$ is the Cuntz-Pimsner algebra $\calO_{X(E)}$. This construction of C*-correspondences associated to topological graphs are introduced by Katsura in \cite{TopGraph}.
	\end{example}

	A result of Katouslis and Kribs shows that the tensor algebra $\calT^+_X$ always sits completely isometrically as a subset of $\calO_X$ \cite[Lemma 3.5]{KK}. Moreover, they show that $\calO_X$ is the C*-envelope of $\calT^+_X$ \cite[Theorem 3.7]{KK}.

\begin{definition}
	Let $(A,X)$ be a C*-correspondence. We define the operator space $S(A,X)$ as the *-closed operator subspace of $\calO_X$ generated by $X$ and $A$.
\end{definition}

An elementary argument shows that $S(A,X)$ sits completely isometrically in both $\calT_X$ and $\calO_X$.

\section{Hyperrigidity of operator spaces $S(A,X)$ }

 In \cite[Theorem 3.1]{KatRam2}, Katsoulis and Ramsey show that to achieve hyperrigidity of a C*-correspondence $X$ that is countably generated over $A$, it is sufficient for the left action of $\calJ_X$ to act non-degenerately $X$ . We show that not only does this condition hold without any assumption on $X$ but that this condition is also necessary. The following two definitions are in \cite{RaeburnThompson}.

 \begin{definition}
   Let $(A,X)$ be a Hilbert $A$-module. We treat the multiplier algebra $M(A)$ as the C*-algebra $\calL(A)$. The Hilbert $M(A)$-module $M(X)$ is defined as follows: As a linear space, $M(X) = \calL(A,X)$. The right action is given by composition and the inner product is given by $\Angle{x,y}:= x^*\circ y$.
 \end{definition}

 If $x \in X$ and $y \in M(X)$ then $\Angle{y,x} \in A$ and if $a \in A$ then $y \cdot a \in X$. If $(A,X)$ is a C*-correspondence and $a \in A$ is such that $\lambda(a) \in K(X)$ then for any $x \in M(X)$, we have $a \cdot x \in X$. In particular, if $a \in \calJ_X$ then $a \cdot x \in X$.

 \begin{definition}
 Let $(A,X)$ be a Hilbert $A$-module. We say that $X$ is countably generated over $A$ if there exists a sequence $(x_n)_{n \geq 1}$ in $M(X)$ for which the $A$-linear span of $(x_n)_n$ is dense in $X$. A standard normalized frame for $(A,X)$ is a sequence $(x_n)_{n \geq 1}$ in $M(X)$ for which for every $x \in X$ we have the identity
 \begin{align*}
  \Angle{x,x} = \sum_{n \geq 1} \Angle{x,x_n}\Angle{x_n,x}\;.
 \end{align*}
 By \cite[Corollary 3.3]{RaeburnThompson}, whenever $X$ is countably generated over $A$, a standard normalized frame for $X$ exists.

 The reconstruction formula \cite[Theorem 3.4]{RaeburnThompson} states that a sequence $(x_n)_{n \geq 1}$ is a standard normalized frame if and only if we have the identity
 \begin{align*}
  x = \sum_{n \geq 1} x_n\Angle{x_n,x}
 \end{align*}
 for every $x \in X$.
 \end{definition}

\begin{lemma}\label{Lemma: approximate unit}
	Suppose that $(A,X)$ is a C*-correspondence. Let $\calM$ denote the space of all countably generated right $A$-submodules of $X$. For each $Y \in \calM$, let $(x_n(Y))_{n \geq 1}$ denote a standard normalized frame for $Y$. Let
	\begin{align*}
		e_n(Y) := \sum_{k=1}^n x_k(Y) \Angle{x_k(Y),\cdot}\;.
	\end{align*}
	The set $(e_n(Y))_{(n,Y) \in \N \times \calM}$ is an approximate unit for $\calK(X)$ in the following sense: if $T \in \calK(X)$ then we have the identity
  \begin{align*}
    \lim_{Y \to \infty} \lim_{n \to \infty} e_n(Y)\cdot T = T\;.
  \end{align*}
\end{lemma}

\begin{proof}
		Let $T \in \calK(X)$. Let $\epsilon > 0$. Suppose that $y_1,\ldots, y_n,z_1,\ldots, z_n \in X$ is such that
		\begin{align*}
			\| T - \sum_{k=1}^n y_k\Angle{z_k,\cdot} \| < \epsilon\;.
		\end{align*}
		Let $S = \sum_k y_k \Angle{z_k,\cdot}$. Consider any $Y \in \calM$ for which $y_k,z_k$ belong to $Y$ for all $k$. For any $x \in X$,
		\begin{align*}
			\sum_k y_k\Angle{z_k, x} \in Y\;.
		\end{align*}
		By the reconstruction formula, we know that
		\begin{align*}
			y_k = \sum_{n \geq 1} x_n(Y)\Angle{x_n(Y), y_k} = \lim_{n \to \infty} e_n(Y)(y_k)\;.
		\end{align*}
		for all $k$. This means in particular, that for $n$ large enough,
		\begin{align*}
			\|e_n(Y)S - S\| < \epsilon \;.
		\end{align*}
		Therefore,
		\begin{align*}
			\|T - e_n(Y)T\| \leq 2\|T -S\| + \|S - e_n(Y)S\| < 3\epsilon\;.
		\end{align*}
		This proves that $e_n(Y)$ is an approximate unit for $\calK(X)$.
\end{proof}

\begin{theorem}\label{theorem: hyperrigidity equivalence}
	Let $(A,X)$ be a C*-correspondence. The following are equivalent:
	\begin{enumerate}
		\item The left action of $\calJ_X$ on $X$ is non-degenerate.
		\item $S(A,X)$ is hyperrigid.
	\end{enumerate}

	\begin{proof}
		First assume that $\calJ_X$ acts on $X$ non-degenerately. We denote by $(i^0,i^1)$ the canonical covariant pair
		\begin{align*}
			(i^0,i^1): (A,X) \to \calO_X\;.
		\end{align*}
		Suppose first that $\calJ_X$ acts on $X$ non-degenerately. Fix any *-homomorphism $\pi: \calO_X \to B(H)$ and suppose that $\phi: \calO_X \to B(H)$ is any cpcc-extension of $\pi|_{S(A,X)}$. By a multiplicative domain argument, it suffices to show that for any $x \in X$, we have
		\begin{align*}
			\phi(\iota^1(x)\iota^1(x)^*) = \phi(\iota^1(x))\phi(\iota^1(x))^*\;.
		\end{align*}
		Let $\calM$ and $x_n(Y),e_n(Y)$ be as in Lemma~\ref{Lemma: approximate unit}. Let
    \begin{align*}
    \varphi_\iota: \calK(X) \to \calO_X : x \Angle{y,\cdot} \mapsto \iota^1(x) \iota^1(y)^*\;.
    \end{align*}
    For any $a \in \calJ_X$, since $\lambda(a)$ is a compact operator, we have
		\begin{align*}
			\iota^0(aa^*) = \varphi_\iota\left(\lim_{Y \to \infty}\lim_{n \to \infty}\lambda(a)\cdot e_n(Y) \lambda(a)^* \right) = \lim_{Y\to \infty}\lim_{n \to \infty} \sum_{k < n} \iota^1(a\cdot x_k(Y)) \iota^1(a \cdot x_k(Y))^* \;.
		\end{align*}
		By the Schwarz inequality,
		\begin{align*}
			\phi(\iota^0(aa^*)) &= {\lim}_{Y} {\lim}_n \sum_{k < n} \phi(\iota^1(a\cdot x_k(Y))\iota^1(a \cdot x_k(Y))^*) \\
			&\geq {\lim}_{Y} {\lim}_n \sum_{k < n} \phi(\iota^1(a \cdot x_k(Y)))\phi(\iota^1(a \cdot x_k(Y)))^* \\
			&= {\lim}_{Y} {\lim}_n \sum_{k < n} \pi(\iota^1(a \cdot x_k(Y)))\pi(\iota^1(a \cdot x_k(Y)))^* \\
			&= \pi(\iota^0(aa^*)) = \phi(\iota^0(aa^*))\;.
		\end{align*}
		From this, we have the identity
		\begin{align*}
			{\lim}_{Y} {\lim}_n \sum_{k < n} \phi(\iota^1(a \cdot x_k(Y))\iota^1(a\cdot x_k(Y))^*) = {\lim}_{Y}{\lim}_n \sum_{k <n} \pi(\iota^1(a \cdot x_k(Y))\iota^1(a \cdot x_k(Y))^*)\;.
		\end{align*}
		By the reconstruction formula, for any $x \in X$ and for any $Y \in \calM$ with $x \in Y$, we have for all $a \in \calJ_X$,
		\begin{align*}
			a\cdot x = \sum_{n \geq 1} a \cdot x_n(Y) \Angle{x_n(Y),x}\;.
		\end{align*}
		Let $\epsilon > 0$. Fix any $Y \in \calM$ for which we have the bound
		\begin{align*}
			0 \leq \sum_{n \geq 1} \phi(\iota^1(a\cdot x_n(Y))\iota^1(a \cdot x_n(Y))^*) - \phi(\iota^1(a\cdot x_n(Y)))\phi(\iota^1(a\cdot x_n(Y)))^* \leq \epsilon 1\;.
		\end{align*}
		Let $\alpha_n = \iota^1(a \cdot x_n(Y))$ and let $\beta_n = \iota^1(a\cdot x \Angle{x,x_n(Y)})$. Observe that
		\begin{align*}
			\iota^1(a\cdot x)\iota^1(a \cdot x)^* = \sum_{n \geq 1} \alpha_n \beta_n^* \;.
		\end{align*}
		Our goal is to show that
		\begin{align*}
			\phi(\iota^1(a \cdot x)\iota^1(a \cdot x)^*) = \phi(\iota^1(a \cdot x)) \phi(\iota^1(a \cdot x))^*\;.
		\end{align*}
		Consider for fixed $n \geq 1$ the $1 \times n$-matrices $A_n = (\alpha_1, \cdots , \alpha_n)$ and $B_n = (\beta_1,\cdots, \beta_n)$. For a positive $M \in M_2(\calO_X)$, let
		\begin{align*}
			P(M):= \left[\begin{array}{c|c}
				&  \\
				 I_{2n} & A_n^* \;\; B_n^* \\
				&  \\\hline
				  \begin{array}{c} A_n \\ B_n \end{array} & M
			\end{array} \right]
		\end{align*}
		The same argument as in \cite[Lemma 3.1]{Paulsen} shows that the matrix $P(M)$ is positive if and only if we have the bound
		\begin{align*}
		\left[\begin{array}{c} A_n \\ B_n \end{array}\right]\left[\begin{array}{cc} A_n^* & B_n^* \end{array} \right] \leq M \;.
		\end{align*}
		Taking
		\begin{align*}
			M = \left[\begin{array}{cc}
					A_nA_n^* & A_n B_n^* \\
					B_nA_n^* & B_nB_n^*
			\end{array} \right]  \;,
		\end{align*}
		we can conclude $P(M)$ is positive in this case. Since $\phi$ is contractive and completely positive, applying the $(2n+2)$-amplification of $\phi$ onto $P(M)$, we get the bound
		\begin{align*}
			\left[\begin{array}{c}\phi(A_n) \\ \phi(B_n) \end{array}\right]\left[\begin{array}{cc} \phi(A_n)^* & \phi(B_n)^* \end{array} \right] \leq \phi(M)\;.
		\end{align*}
		That is, the matrix
		\begin{align*}
			\left[\begin{array}{c c} \phi(A_nA_n^*) - \phi(A_n)\phi(A_n)^* & \phi(A_n B_n^*) - \phi(A_n)\phi(B_n^*) \\
			\phi(B_n A_n^*) - \phi(B_n)\phi(A_n)^* & \phi(B_nB_n^*) - \phi(B_n)\phi(B_n)^*
			\end{array}\right]
		\end{align*}
		is positive. Since the $(1,1)$ corner of this matrix is at most $\epsilon$, we get positivity of the matrix
	\begin{align*}
			\left[\begin{array}{c c} \epsilon I_2 & \phi(A_n B_n^*) - \phi(A_n)\phi(B_n^*) \\
			\phi(B_n A_n^*) - \phi(B_n)\phi(A_n)^* & \phi(B_nB_n^*) - \phi(B_n)\phi(B_n)^*
			\end{array}\right]\;.
		\end{align*}
		In particular, we have the bound
		\begin{align}
			\| \phi(A_nB_n^*) - \phi(A_n)\phi(B_n)^*\|^2 &\leq \epsilon\|\phi(B_nB_n^*) - \phi(B_n)\phi(B_n)^*\| \nonumber \\
			&\leq 2\epsilon\|B_nB_n^*\| \label{equation: main}\;.
		\end{align}
		A calculation shows that
		\begin{align*}
			\| B_n B_n^* \| &= \left\|\sum_{k \leq n} \iota^1(a\cdot x)\iota^0(\Angle{x, x_k(Y)}\Angle{x_k(Y), x})\iota^1(a \cdot x)^* \right\| \\
      &= \left\|\iota^1(a \cdot x)\;\iota^0\left( \sum_{k \leq n} \Angle{x,x_k(Y)}\Angle{x_k(Y),x}\right) \iota^1(a \cdot x)^* \right\| \\
			&\leq \|a \cdot x\|^2 \left\|\sum_{k \leq n} \Angle{x,x_k(Y)}\Angle{x_k(Y),x} \right\|\;.
		\end{align*}
		Since the sequence $x_n(Y)$ is a standard normalized frame, we have the identity
		\begin{align*}
			\|x\|^2 = \left\|\sum_{k \geq 1} \Angle{x,x_k(Y)}\Angle{x_k(Y),x}\right\| \;.
		\end{align*}
		Therefore, we have the inequality
		\begin{align*}
			\|B_nB_n^*\| \leq \|a \cdot x\|^2 \|x\|^2\;
		\end{align*}
		for any $n$. As well, a calculation shows that
		\begin{align*}
				\lim_{n \to \infty} \phi(A_nB_n^*) - \phi(A_n)\phi(B_n)^*	=& \lim_{n \to \infty} \sum_{k \leq n } \phi(\iota^1(a \cdot x_n(Y)\Angle{x_n(Y),x}) \iota^1(a\cdot x)^*) \\
				&-\phi(\iota^1(a \cdot x_n(Y)\Angle{x_n(Y),x}))\phi(\iota^1(a\cdot x))^* \\
				=& \phi(\iota^1(a \cdot x) \iota^1(a \cdot x)^*) - \phi(\iota^1(a \cdot x)) \phi(\iota^1(a \cdot x))^* \;.
		\end{align*}
		The above calculation with equation~\ref{equation: main} give us the bound
		\begin{align*}
			\|\phi(\iota^1(a \cdot x) \iota^1(a \cdot x)^*) - \phi(\iota^1(a \cdot x)) \phi(\iota^1(a \cdot x))^* \|^2 &= \lim_{n \to \infty} \| \phi(A_nB_n^*) - \phi(A_n)\phi(B_n)^*\|^2\\
			&\leq 2 \epsilon \|a \cdot x\|^2\|x\|^2\;.
		\end{align*}
		Since this identity is independent of the choice of $Y$ and $\epsilon$, we may conclude that for any $a \in \calJ_X$ and for any $x \in X$, the element $\iota^1(a \cdot x)$ belongs to the multiplicative domain of $\phi$. Since $\calJ_X$ acts non-degenerately on $X$, this proves that any element of $\iota^1(X)$ belongs to the multiplicative domain of $\phi$, showing hyperrigidity.

		For the converse, assume that $\calJ_X$ does not act on $X$ non-degenerately. Fix a faithful covariant representation $(\pi^0, \pi^1): (A,X) \to B(H)$. Let $N \subset \calJ_{X,+}$ form a contractive approximate unit for $\calJ_X$ under the ordering induced by the positive operators. Define operators $P = \lim_{a \in N} \pi^0(a)$ and $Q = 1-P$ where the limit is taken in the strong operator topology on $B(H)$. For any isometry $V \in B(\calK)$, let $(\tau^0,\tau_V^1) : (A,X) \to B(H \otimes \calK)$ be the following pair of maps
		\begin{align*}
			\tau^0 &: A \to B(H \otimes \calK) : a \mapsto \pi^0(a) \otimes I \\
			\tau_V^1 &: X \to B(H \otimes \calK) : x \mapsto P\pi^1(x) \otimes I + Q\pi^1(x) \otimes V\;.
		\end{align*}
		It is immediate that $\tau^0$ is a *-homomorphism and that $\tau_V^1$ is linear. For any $a \in A$ and $x \in X$, first observe that since $P$ is the projection which generates the ideal $\pi^0(\calJ_X)$ in $\pi^0(A)$, that $P$ commutes with $\pi^0(a)$. Thus,
		\begin{align*}
			\tau^0(a)\tau_V^1(x) &= (\pi^0(a) \otimes I)(P\pi^1(x)\otimes I + Q\pi^1(x) \otimes V) \\
				&= P(\pi^0(a)\pi^1(x))\otimes I + Q\pi^0(a)\pi^1(x) \otimes V \\
				&= P\pi^1(a\cdot x) \otimes I + Q \pi^1(a \cdot x) \otimes V = \tau^1(a \cdot x)\;.
		\end{align*}
		As well, for $x, y \in X$, we have
		\begin{align*}
			\tau_V^1(x)^*\tau_V^1(x) &=(P \pi^1(x) \otimes I + Q \pi^1(x) \otimes V)^*(P \pi^1(y) \otimes I + Q \pi^1(y) \otimes V) \\
			&= \pi^1(x)^*P\pi^1(y) \otimes I + \pi^1(x)^*Q\pi^1(y) \otimes I \\
			&= \pi^1(x)^*(P + Q)\pi^1(y) \otimes I = \pi^0(\Angle{x,y}) \otimes I \\
			&= \tau^0(\Angle{x,y})\;.
		\end{align*}
		This is therefore a Toeplitz representation for $(A,X)$. To see that this representation is covariant, let $a \in \calJ_X$. Since $\lambda(a) \in \calK(X)$, for $\epsilon > 0$, let $x_1,\ldots, x_n, y_1,\ldots, y_n \in X$ such that for any contraction $z \in X$, we have
		\begin{align*}
			\left\|a \cdot z - \sum_{k=1}^n x_k \Angle{y_k,z} \right\| < \epsilon \;.
		\end{align*}
		For any $b \in N$, we have
		\begin{align*}
			\left\|b a b  \cdot z - \sum_{k=1}^n b\cdot x_k \Angle{b \cdot y_k,z} \right\| < \epsilon \;.
		\end{align*}
		In particular, $\lambda(b a b)$ is within $\epsilon$ of the compact operator $\sum_k bx_k \Angle{by_k, \cdot}$. Let
    \begin{align*}
      \phi_V : \calK(X) \to B(H) : x_i \Angle{y_i,\cdot} \mapsto \tau^1_V(x_i)\tau^1_V(y_i)^*\;.
    \end{align*}
    A calculation shows
		\begin{align*}
			&\phi_V\left({\sum}_k bx_k \Angle{by_k, \cdot} \right) = {\sum}_k \tau^1(b x_k) \tau^1(b y_k)^* \\
			&= {\sum}_k (P \pi^1(bx_k) \otimes I + Q \pi^1(b x_k) \otimes V)(P \pi^1(b y_k) \otimes I + Q \pi^1(b y_k) \otimes V)^* \\
			&= {\sum}_k (P \pi^1(b x_k) \otimes I )(P \pi^1(b y_k) \otimes I)^* \\
			&= {\sum}_k (\pi^1(b x_k) \otimes I)(\pi^1(b y_k) \otimes I)^* \\
			&= \left( {\sum}_k \pi^1(b x_k) \pi^1(b y_k)^* \right) \otimes I \;.
		\end{align*}
		For any $b \in N$,
		\begin{align*}
			\left\|\phi_V(\lambda(b a b)) - \pi^0(bab) \otimes I  \right\| \leq& \left\| \phi_V(\lambda(b a b)) - \phi_V\left({\sum}_k bx_k \Angle{by_k, \cdot} \right)   \right\| \\
			&+ \left\|\phi_V\left({\sum}_k bx_k \Angle{by_k, \cdot} \right) - \pi^0(bab) \otimes I \right\| \\
			<& 2\epsilon \;.
		\end{align*}
		Since this is true for arbitrary $\epsilon>0$, we conclude that $\phi_V(\lambda(b a b)) = \tau^0(b a b)$ for all $b \in N$. Since $N$ is an approximate unit for $\calJ_X$ and $a \in \calJ_X$, we have $\phi_V(\lambda(a)) = \tau^0(a)$.

		Let us fix the unilateral shift $V \in B(\ell^2(\Z_+))$ and the bilateral shift $U \in B(\ell^2(\Z))$. Let $\Phi: B(\ell^2(\Z)) \to B(\ell^2(\Z_+))$ be the ucp map given by restriction. The diagram
		\begin{center}
			\begin{tikzcd}[row sep=large, column sep=large]
				& B(H \otimes \ell^2(\Z)) \arrow{d}{\Phi} \\
				\calO_X \arrow{ur}{\tau^0 \times \tau^1_U} \arrow{r}{\tau^0 \times \tau^1_V} & B(H \otimes \ell^2(\Z_+))
			\end{tikzcd}
		\end{center}
		commutes. So long as we can show $Q \pi^1(X) \neq 0$, we are done, since $\Phi \circ (\tau^0 \times \tau^1_U) \neq \tau^0 \times \tau^1_V$ but agree on $S(A,X)$. Suppose that $Q \pi^1(X) = 0$ in order to derive a contradiction. Since $P + Q = I$, this means that $P\pi^1(x) = \pi^1(x)$ for every $x \in X$. If $\calJ_X$ acts on $X$ degenerately, then by taking a subnet if necessary, there is some $\epsilon > 0$ and some $x \in X$ so that for every $b \in N$, there is some unit vector $h_b \in H$ for which we have the identity
		\begin{align*}
 			\Angle{(\pi^1(x)^*\pi^1(x) - \pi^1(x)^*\pi^0(b)\pi^1(x))h_b,h_b} \geq \epsilon\;.
		\end{align*}
		If $a \geq b$ in $N$ then we have the identity
		\begin{align*}
			\Angle{(\pi^1(x)^*\pi^1(x) - \pi^1(x)^*\pi^0(b)\pi^1(x))h_a, h_a} &\geq \Angle{(\pi^1(x)^*\pi^1(x) - \pi^1(x)^*\pi^0(a)\pi^1(x))h_a,h_a} \\
			&\geq \epsilon\;.
		\end{align*}
		If we could replace the net $(h_b)_{b \in N}$ with a fixed vector $h_b = h \in H$ for all $b$ then we may conclude from the above inequality that  $P\pi^1(x) \neq \pi^1(x)$ and we would have our contradiction.

    In order to guarantee that a vector $h \in H$ as above exists, we need to fix a specific faithful representation. Take any non-principal ultrafilter $\calU$ over $N$ containing the set
		\begin{align*}
			S := \{\{a \in N : a \geq b \}: b \in N \}\;.
		\end{align*}
		Such an ultrafilter exists since $S$ has the finite intersection property. Consider the covariant pair
		\begin{align*}
			(\cl{\pi}^0, \cl{\pi}^1):(A,X) \to B(H^\calU)
		\end{align*}
		so that $\cl{\pi}^0(a)(\lim_\calU k_b) = \lim_\calU \pi^0(a)\cdot k_b$ and $\cl{\pi}^1(x)(\lim_\calU k_b) = \lim_\calU \pi^1(x)\cdot k_b$. Replacing $(\pi^0,\pi^1)$ with $(\cl{\pi}^0,\cl{\pi}^1)$ and taking $h = \lim_\calU h_b$ will do.
	\end{proof}
\end{theorem}

\begin{corollary}
  Let $G$ be a locally compact group and let $(A,X)$ be a $G$-C*-correspondence with $\calJ_X \acts X$ non-degenerate. We have the isomorphisms
  \begin{enumerate}
    \item $\calO_X \rtimes_r G = \calO_{X \rtimes_r G}$ and
    \item $\calO_X \rtimes G = \calO_{X \Hat{\rtimes} G}$
  \end{enumerate}
  where $(A \Hat{\rtimes}G , X \Hat{\rtimes} G)$ is the C*-correspondence given by the completion of $(C_c(G,A), C_c(G,X))$ in $\calO_X \rtimes G$.
\end{corollary}

\begin{proof}
  By \cite[Theorem 3.10]{KatRam2}, our isomorphisms follow from hyperrigidity of $(A,X)$. By Theorem~\ref{theorem: hyperrigidity equivalence} we have the result.
\end{proof}

As an application, we will characterize the topological graphs with range map $r$ open for which the associated space $S(C_0(E^0), X(E))$ is hyperrigid. This generalizes a result of Kakariadis \cite[Theorem 3.3]{Kakariadis} and Dor-On and Salomon \cite[Theorem 3.5]{DorOnSalomon} that give a characterization for $E$ discrete. First a bit of notation: let $E^0_{\fin}$ be the open subset of $E^0$ for which we have the identity
\begin{align*}
	C_0(E^0_{\fin}) = \lambda^{-1}(\calK(X(E)))\;.
\end{align*}

The kernel of $\lambda$ consists of those elements $f \in C_0(E^0)$ for which $f|_{r(E^1)} = 0$. Thus,
\begin{align*}
	\ker \lambda = C_0(E^0 \setminus \cl{r(E^1)})\;.
\end{align*}
This implies that $\calJ_{X(E)} = C_0(E^0_{\fin} \cap \Int(\cl{r(E^1)}))$. Let $Y = \Int(\cl{r(E^1)})$. Assume that $E^0_{\fin} \cap Y$ is dense in $Y$. I claim that $\calJ_{X(E)} X(E) = X(E)$. Let $\phi_i$ be an approximate unit for $C_0(E^0_{\fin} \cap Y)$. For any $f \in C_c(E^1)$, I claim that $\phi_i \cdot f$ converges to $f$. Consider the positive function $\calF_i = \Angle{f - \phi_i \cdot f, f - \phi_i \cdot f}$. Observe that as $f$ is compactly supported that all $\calF_i$ are supported on a compact set $K$. As well, $\calF_i(x)$ is a decreasing net for all $x \in E^0$. By the uniform limit theorem, the function
\begin{align*}
	\calF : E^0 \to \C: x \mapsto \lim_{i \to \infty} \calF_i(x)
\end{align*}
is continuous and compactly supported. We need to show that $\calF = 0$. If not, there is some open set $U \subset E^0$ for which $\calF|_U > 0$. If $x \in U$ then for any $e \in s^{-1}(x)$, $r(e) \not\in E^0_{\fin}$. That is, if $x \in r(s^{-1}(U))$ then $x \not\in E^0_{\fin}$. Assume that $r$ is open. That $r(s^{-1}(U))$ is an open subset of $Y$ and that $r(s^{-1}(U)) \cap E^0_{\fin} = \varnothing$ is a contradiction on the density of $E^0_{\fin} \cap Y$ in $Y$. Thus we have $\calJ_{X(E)} X(E) = X(E)$.

If $E^0_{\fin} \cap Y$ is not dense in $Y$ then there is some open subset $U$ of $Y$ so that $U \cap E^0_{\fin} = \varnothing$. Consider any non-zero function $f \in C_c(E^1)$ supported on $r^{-1}(U)$. If $\calJ_{X(E)}$ acts non-degenerately on $X(E)$, then by Cohen's factorization theorem, there is some $x \in X(E)$ and some $g \in \calJ_{X(E)}$ for which $g \cdot x = f$. Say $f_i \in C_c(E^1)$ for which $\lim_i f_i = x$. For any point $e \in E^1$, if $f(e) \neq 0$ then $r(e) \in U$. This implies that $g(r(e)) = 0$. For any $i$,
\begin{align*}
	\Angle{g \cdot f_i, f}: x \mapsto \sum_{e \in E^1: s(e) = x} \cl{g(r(e))}f(e)\cl{f_i(e)}  = 0\;.
\end{align*}
Thus we have $\Angle{f,f} = \lim_i \Angle{g \cdot f_i,f} = 0$ -- a contradiction.

All this proves:
\begin{theorem}
	Let $E$ be a topological graph and let $r$ be open. The following are equivalent:
	\begin{enumerate}
		\item The space $S(C_0(E^0), X(E))$ is hyperrigid.
		\item The set $E^0_{\fin}$ is dense in $E^0$.
	\end{enumerate}
\end{theorem}

\begin{proof}
	Let $Y = \Int(\cl{r(E^1)})$. By the above argument, hyperrigidity of $S(C_0(E^0), X(E))$ is equivalent to density of $E^0_{\fin} \cap Y$ in $Y$. To finish the argument, suppose that $E^0_{\fin} \cap Y$ is dense in $Y$. If $x$ is a point in $E^1 \setminus \cl{r(E^1)}$ then there is a non-negative function $f$ supported outside of $\cl{r(E^1)}$ for which $f(x) =1$. Since $\lambda(f) = 0$, we must conclude that $x \in E^0_{\fin}$. In particular, whenever $U$ is an open set in $E^0$ for which $U \cap E^0_{\fin} = \varnothing$ then we must have $U \subset Y$. By our assumption, we must have $U = \varnothing$.
\end{proof}

\section*{Acknowledgements}
The author would like to thank Ken Davidson and Matt Kennedy for their careful reading of this paper and their many helpful comments. The author is supported by an NSERC CGS-D scholarship.

\end{document}